\documentclass{article}

\usepackage{amsmath}
\usepackage{amssymb}
\usepackage{graphicx} 
\usepackage{amsthm}
\usepackage[margin=1in]{geometry}

\newtheorem{definition}{Definition}[section]
\newtheorem{lemma}{Lemma}[section]

\title{A Short Report on Importance Sampling for Rare Event Simulation in Diffusions}
\author{
Zhiwei Gao\\
Division of Applied Mathematics, Brown University\\
\texttt{zhiwei\_gao@brown.edu}
}

\date{}

\begin{document}

\maketitle
\begin{abstract}
    In this manuscript, we investigate importance sampling methods for rare-event simulation in diffusion processes. We show, from a large-deviation perspective, that the resulting importance sampling estimator is log-efficient. This connection is established via a stochastic optimal control formulation, and the associated Hamilton--Jacobi--Bellman (HJB) equation is derived using dynamic programming. To approximate the optimal control, we adopt a spectral parameterization and employ the cross-entropy method to estimate the parameters by solving a least-squares problem. Finally, we present a numerical example to validate the effectiveness of the cross-entropy approach and the efficiency of the resulting importance sampling estimator.
\end{abstract}

\section{Problem setup}
The problem here is to estimate \emph{quantities of interest} in the following form  \cite{vanden2012rare} 
\begin{equation}
    \rho = \mathbb{E}\!\left[e^{-\frac{1}{\epsilon}g(X^{\epsilon})}\right],
\end{equation}
where $g: \mathcal{C}([0, T]; \mathbb{R}^{d})\to \mathbb{R}$ is a bounded continuous functional and $X^{\epsilon}$ is the strong solution of the following SDE:
\begin{equation}
\begin{cases}
dX^\varepsilon(s)
  = b\big(X^\varepsilon(s), s\big)\,ds
    + \sqrt{\varepsilon}\,\sigma\big(X^\varepsilon(s), s\big)\,dW(s),
  & s \in [0,T], \\[4pt]
X^\varepsilon(0) = x_0, &
\end{cases}
\end{equation}
with a $d$-dimensional Brownian motion in some probability space $(\Omega, \mathcal{F}, \mathcal{F}_{t}, P)$ that satisfies the usual conditions. Typically, we will take $\Omega = \mathcal{C}([0, T]; \mathbb{R}^{d})$ in this manuscript. To ensure existence and uniqueness of the solution, the drift term $b: \mathbb{R}^{d}\times[0, T]\rightarrow \mathbb{R}^{d}$ and the diffusion term $\sigma: \mathbb{R}^{d}\times [0, T]\rightarrow \mathbb{R}^{d\times d}$ satisfy the corresponding conditions, i.e.,
    \begin{align}
|b(x,t)| + |\sigma(x,t)| &\le C(1 + |x|), \\
|b(x,t) - b(y,t)| + |\sigma(x,t) - \sigma(y,t)| &\le D |x - y|,
\end{align}
for some constant $C$ and $D$. Typically, we are interested in the cases where $\epsilon \ll 1$, which results in small perturbations of the diffusion and thus rare event simulation. 

\section{Monte Carlo simulation}
To approximate the value in (1), one common way is to use the Monte Carlo simulation, i.e, 
\begin{equation}
    \tilde{\rho}_{MC} \approx \frac{1}{N}\sum_{i=1}^{N}e^{-\frac{1}{\epsilon}g(X^{\epsilon}_{i})},
\end{equation}
by repeatedly solving the corresponding SDE. This estimator is unbiased, i.e, 
\begin{equation}
    \mathbb{E}[\tilde{\rho}_{MC}] = \rho.
\end{equation}
The variance is 
\begin{equation*}
    \operatorname{Var}(\tilde{\rho}_{MC})
= \frac{1}{N} \Big( \mathbb{E}\big[e^{-\tfrac{2}{\varepsilon} g(X^\varepsilon)}\big]
 - \rho^{2} \Big).
\end{equation*}
Thus, the corresponding coefficient of variation is 
\begin{equation}
    \delta(\tilde{\rho}_{MC}) = \frac{\sqrt{\operatorname{Var}(\tilde{\rho}_{MC})}}{\rho} = \frac{1}{\sqrt{N}}
\sqrt{
    \frac{\mathbb{E}\big[e^{-\tfrac{2}{\varepsilon} g(X^\varepsilon)}\big]}
         {\big(\mathbb{E}\big[e^{-\tfrac{1}{\varepsilon} g(X^\varepsilon)}\big]\big)^2}
    - 1
}.
\end{equation}
The MC estimator seems to be very efficient when $N$ increases. However, according to Varadhan’s lemma, under some general conditions, we have 
\begin{align}
\lim_{\varepsilon \to 0} \varepsilon 
    \log \mathbb{E}\bigl[e^{-\tfrac{1}{\varepsilon} g(X^\varepsilon)}\bigr]
&= - \inf_{\substack{\varphi \in \mathcal{C}([0,T];\mathbb{R}^{d}) \\ \varphi(0)=x_0}}
    \{ I(\varphi) + g(\varphi) \}
:= -\gamma_1, \\[6pt]
\lim_{\varepsilon \to 0} \varepsilon 
    \log \mathbb{E}\bigl[e^{-\tfrac{2}{\varepsilon} g(X^\varepsilon)}\bigr]
&= - \inf_{\substack{\varphi \in \mathcal{C}([0,T];\mathbb{R}^{d}) \\ \varphi(0)=x_0}}
    \{ I(\varphi) + 2 g(\varphi) \}
:= -\gamma_2,
\end{align}
where $\mathcal{C}([0,T];\mathbb{R}^{d})$ is the set of all \emph{continuous} functions from $[0, T]$ to $\mathbb{R}^{d}$. 
In particular, the rate functional $I(\varphi)$ is finite only for absolutely continuous paths $\varphi\in \mathcal{AC}([0,T];\mathbb{R}^d)$; for non-$AC$ paths we set $I(\varphi)=+\infty$.
And $I(\varphi)$ is the rate functional for the process $X^{\epsilon}$ defined by 
\begin{equation}
    I(\varphi)
    =
    \begin{cases}
    \displaystyle
    \inf_{\substack{u \in L^2([0,T];\mathbb{R}^d)\\
                      \dot{\varphi}(t)= b+\sigma\,u}}
        \int_0^T \frac{1}{2}\,|u(t)|^2\,\mathrm{d}t,
        & \varphi\in AC([0,T];\mathbb{R}^d),\\[2ex]
        +\infty, & \text{otherwise.}
    \end{cases}
\end{equation}
Based on Jensen's inequality, we have $\gamma_{2} \leq 2\gamma_{1}$. And we can write 
\begin{equation}
\begin{split}
    \mathbb{E}\bigl[e^{-\tfrac{1}{\varepsilon} g(X^\varepsilon)}\bigr] &= e^{-\frac{\gamma_{1} + o(1)}{\epsilon}},\\ 
    \mathbb{E}\bigl[e^{-\tfrac{2}{\varepsilon} g(X^\varepsilon)}\bigr] &= e^{-\frac{\gamma_{2} + o(1)}{\epsilon}}.
\end{split}
\end{equation}
Thus, we have 
\begin{equation}
    \delta(\tilde{\rho}_{MC}) = \frac{1}{\sqrt{N}}\sqrt{e^{\frac{2\gamma_{1}-\gamma_{2} + o(1)}{\epsilon}}-1}.
\end{equation}
It is easy to discover that when $\epsilon$ approaches 0, the number of samples $N$ needs to grow exponentially to maintain the error. A desirable situation is that $\gamma_{2} = 2\gamma_{1}$, in which case the leading exponential growth in the coefficient of variation cancels out. But to let $\delta(\tilde{\rho}_{MC})\rightarrow 0$ when $\epsilon\to 0$, we still need to erase second order terms, which is quite difficult. Nevertheless, to let $\gamma_{2} = 2\gamma_{1}$, we need to adopt importance sampling to control the second order moment, which will be introduced in the following section.

\section{Importance sampling}

To implement importance sampling, we denote by $\mathbb{P}^\varepsilon$ the path measure induced by (2) on $\Omega=\mathcal{C}([0,T];\mathbb{R}^d)$. 
Suppose the proposal measure used for importance sampling is
$\mathbb{Q}\ll \mathbb{P}^{\epsilon}$ and set
$
Z = \frac{d\mathbb{P}^\varepsilon}{d\mathbb{Q}}.
$
Then the importance sampling estimator for (1) is \cite{zhang2014applications}
\begin{equation}
  \tilde{\rho}_{\mathrm{IS}}
  = \frac{1}{N} \sum_{i=1}^N
      e^{-\tfrac{1}{\varepsilon} g(\tilde X_i^\varepsilon)}\, Z(\tilde X_i^\varepsilon),
  \label{eq:IS-estimator}
\end{equation}
where $\{\tilde X_i^\varepsilon\}_{i=1}^N$ are i.i.d.\ samples drawn from $\mathbb{Q}$. The importance sampling estimator is also unbiased and the variance is given by 
\begin{equation}
\begin{split}
\operatorname{Var}(\tilde{\rho}_{\mathrm{IS}})
= \frac{1}{N}\left(
    \mathbb{E}_{\mathbb{Q}}\Big[
        e^{-\tfrac{2}{\varepsilon} g(\tilde X^\varepsilon)}
        Z(\tilde X^\varepsilon)^2
    \Big]
    - \rho^2
\right).
\end{split}
\end{equation}
The error is given by 
\begin{equation}
\delta(\tilde{\rho}_{\mathrm{IS}})
= \frac{\sqrt{\operatorname{Var}(\tilde{\rho}_{\mathrm{IS}})}}{\rho}
= \frac{1}{\sqrt{N}}
  \sqrt{
    \frac{
      \mathbb{E}_{\mathbb{Q}}\big[
        e^{-\tfrac{2}{\varepsilon} g(\tilde X^\varepsilon)}
        Z(\tilde X^\varepsilon)^2
      \big]
    }{
      \Big(
        \mathbb{E}_{\mathbb{Q}}\big[
          e^{-\tfrac{1}{\varepsilon} g(\tilde X^\varepsilon)}
          Z(\tilde X^\varepsilon)
        \big]
      \Big)^2
    }
    - 1 }.
\end{equation}
Actually, based on Jensen's inequality,
\begin{equation}
\begin{aligned}
\limsup_{\varepsilon \to 0}
    -\varepsilon \log \mathbb{E}_{\mathbb{Q}}\Big[
        e^{-\tfrac{2}{\varepsilon} g(\tilde X^\varepsilon)}
        Z(\tilde X^\varepsilon)^{2}
    \Big]
&\le
2 \lim_{\varepsilon \to 0}
    -\varepsilon \log \mathbb{E}\Big[
        e^{-\tfrac{1}{\varepsilon} g(X^\varepsilon)}
    \Big] \\
&= 2\gamma_{1}.
\end{aligned}
\end{equation}
Based on this, we may find some measure $\mathbb{Q}$ such that the equality holds and then the IS estimator is log-efficient.
\begin{definition}
   Let the ratio be 
   \[
   R(\epsilon) =
   \frac{
      \mathbb{E}_{\mathbb{Q}}\big[
        e^{-\tfrac{2}{\varepsilon} g(\tilde X^\varepsilon)}
        Z(\tilde X^\varepsilon)^2
      \big]
    }{
      \Big(
        \mathbb{E}_{\mathbb{Q}}\big[
          e^{-\tfrac{1}{\varepsilon} g(\tilde X^\varepsilon)}
          Z(\tilde X^\varepsilon)
        \big]
      \Big)^2}.
      \]
      The importance sampling estimator is log-efficient if
      \begin{equation}
          \limsup_{\epsilon\to 0}\ \epsilon \log R(\epsilon) = 0.
      \end{equation}
\end{definition}
    
\section{Importance sampling and stochastic optimal control}
The following task is to construct an optimal proposal measure $\mathbb Q$ that
minimizes the variance of the estimator. In this manuscript, we restrict
attention to functionals $g$ that depend only on the terminal state
$X_T^\varepsilon$, i.e., $g = g(X_T^\varepsilon)$. Moreover, for simplicity, we consider one-dimensional cases, where higher dimensions are similar. In this case, the zero-variance
change of measure satisfies
\begin{equation}\label{eq:zero-variance-density}
  \frac{d\mathbb Q^*}{d\mathbb P^\varepsilon}
    = \frac{e^{-g(X_T^\varepsilon)/\varepsilon}}{\rho}
    =: L_T^* .
\end{equation}
Under $\mathbb{Q}^*$, the corresponding IS weight is $Z^*=\frac{d\mathbb{P}^\varepsilon}{d\mathbb{Q}^*}=(L_T^*)^{-1}$ and the estimator has zero variance.

Define $L_t := \mathbb E[L_T^* \mid \mathcal F_t]$ for $0 \le t \le T$. Then
$(L_t)_{0\le t\le T}$ is a positive martingale with $L_0 = 1$ by Doob's construction. By the martingale
representation theorem, there exists a progressively measurable process $u^{*}$
such that
\begin{equation}\label{eq:Z-exponential}
  L_t
  = \exp\Bigg(
      -\frac{1}{\sqrt{\epsilon}}\int_0^t u_s^{*} \, dW_s
      - \frac{1}{2\epsilon} \int_0^t \lvert u_s^{*} \rvert^2 ds
    \Bigg), \qquad 0 \le t \le T.
\end{equation}
Hence, by Girsanov's theorem, finding the optimal proposal measure
$\mathbb Q^*$ is equivalent to finding an optimal drift control $u$. Define the control space to be 
\begin{equation}
   \mathcal U
:= \Big\{ u:\Omega\times [0, T]\to\mathbb R^d  \,\Big|\,
u \text{ is progressively measurable w.r.t. } \mathcal F_t,
\ \mathbb E\!\int_0^T |u(s)|^2 ds <\infty \Big\}.
\end{equation}
Then the stochastic optimal control (risk-sensitive) representation is given by 
\begin{equation}
    V(x, t) = \inf_{u\in \mathcal{U}}\mathbb E_{\mathbb{P}^\varepsilon}\!\left[\exp\!\left(- \tfrac{2}{\varepsilon} g(\hat{X}_T^\varepsilon)
  +\tfrac{1}{\varepsilon}\int_t^T |u_s(\hat{X}_s^\varepsilon)|^{2}\,ds\right)\Big|\, \hat{X}_t^\varepsilon = x\right].
\end{equation}
To justify the above representation, we need to introduce the following lemma:

\begin{lemma} \cite{dupuis2012importance} Let $X^{\epsilon}$ be the solution of (2). Given $u\in \mathcal{U}$, the induced measure $\mathcal{Q}$ is given by Girsonov's theorem, i.e., \begin{equation} \frac{d\mathbb{Q}}{d\mathbb{P}} = \exp\left(\frac{1}{\sqrt{\epsilon}}\int_{t}^{T}u_{s}dW_{s} - \frac{1}{2\epsilon}\int_{t}^{T}|u_{s}|^{2}ds\right), \end{equation} and $B_{s} = W_{s} - \frac{1}{\sqrt{\epsilon}}\int_{t}^{s}u_{r}dr$ is a standard Brownian motion under $\mathbb{Q}$. Similarly, we introduce the control $-u$ and denote the induced measure as $\mathbb{Q}^{-u}$, where the new standard Brownian motion is given by $B^{-u}_{s} = W_{s} + \frac{1}{\sqrt{\epsilon}}\int_{t}^{s}u_{r}dr$. Then we have \begin{equation} E_{\mathbb{Q}}\bigl[e^{- \tfrac{2}{\varepsilon} g(X_T^\varepsilon)} Z_T^2 \,\big|\, X_t^\varepsilon = x\bigr] = \mathbb{E}_{\mathbb{P}}[\exp(-\frac{2}{\epsilon}g(\hat{X}^{\epsilon}_{T})+\frac{1}{\epsilon}\int_{t}^{T}|u_{s}(\hat{X}^{\epsilon}_{s})|^{2}ds)|\hat{X}^{\epsilon}_{t} = x], \end{equation} where $\hat{X}^{\epsilon}_{s}$ is the solution of \begin{equation} d\hat{X}^{\epsilon}_{s} = (b(\hat{X}^{\epsilon}_{s}, s) - \sigma(\hat{X}^{\epsilon}_{s} ,s)u_{s}(\hat{X}^{\epsilon}_{s} ))ds + \sqrt{\epsilon}\sigma(\hat{X}^{\epsilon}_{s} , s) dW_{s}, t\leq s\leq T. \end{equation} with initial condition $\hat{X}^{\epsilon}_{s} = x$ . \end{lemma} \begin{proof} Since \begin{equation*} \begin{split} dX^\varepsilon(s) &= b\big(X^\varepsilon(s), s\big)\,ds + \sqrt{\varepsilon}\,\sigma\big(X^\varepsilon(s), s\big)\,dW_{s}\\ & = \left(b\big(X^\varepsilon(s), s\big) -\sigma(X^{\epsilon}_{s} ,s)u_{s}(X^{\epsilon}_{s} )\right) \,ds + \sqrt{\varepsilon}\,\sigma\big(X^\varepsilon(s), s\big)\,dB^{-u}_{s}. \end{split} \end{equation*} Based on the uniqueness of the strong solution, the distribution of $X^{\epsilon}_{s}$ under $B^{-u}_{s}$ is the same distribution as $\hat{X}^{\epsilon}_{s}$ under $W_{s}$. Therefore, we have \begin{equation} \begin{split} \mathbb{E}&_{\mathbb{P}^{\epsilon}}[\exp(-\frac{2}{\epsilon}g(\hat{X}^{\epsilon}_{T})+\frac{1}{\epsilon}\int_{t}^{T}|u_{s}(\hat{X}^{\epsilon}_{s})|^{2}ds)|\hat{X}^{\epsilon}_{t} = x] \\ & = \mathbb{E}_{\mathbb{Q}^{-u}}\left[\exp(-\frac{2}{\epsilon}g(X^{\epsilon}_{T})+\frac{1}{\epsilon}\int_{t}^{T}|u_{s}(X^{\epsilon}_{s})|^{2}ds)|X^{\epsilon}_{t} = x\right]\\ & = \mathbb{E}_{\mathbb{P}^{\epsilon}}\left[\left(\exp(-\frac{2}{\epsilon}g(X^{\epsilon}_{T})+\frac{1}{\epsilon}\int_{t}^{T}|u_{s}(X^{\epsilon}_{s})|^{2}ds)\right)\frac{d\mathbb{Q}^{-u}}{d\mathbb{P}^{\epsilon}}|X^{\epsilon}_{t} = x\right]\\ & = \mathbb{E}_{\mathbb{P}^{\epsilon}}\left[\left(\exp(-\frac{2}{\epsilon}g(X^{\epsilon}_{T})+\frac{1}{2\epsilon}\int_{t}^{T}|u_{s}(X^{\epsilon}_{s})|^{2}ds - \frac{1}{\sqrt{\epsilon}}\int_{t}^{T}u_{s}(X^{\epsilon}_{s})dW_{s}\right)|X^{\epsilon}_{t} = x\right]\\ & = \mathbb{E}_{\mathbb{P}^{\epsilon}}\left[\exp(-\frac{2}{\epsilon}g(X^{\epsilon}_{T}))\frac{d\mathbb{P}^{\epsilon}}{d\mathbb{Q}}|X^{\epsilon}_{t} = x\right]\\ & = \mathbb{E}_{\mathbb{Q}}\left[\exp(-\frac{2}{\epsilon}g(X^{\epsilon}_{T}))\left(\frac{d\mathbb{P}^{\epsilon}}{d\mathbb{Q}}\right)^{2}|X^{\epsilon}_{t} = x\right], \end{split} \end{equation} which finishes the proof. \end{proof}

Using lemma 4.1, the original problem is transformed into a stochastic optimal control problem
\begin{equation}
    V(x, t) = \inf_{u\in \mathcal{U}}\mathbb{E}_{\mathbb{P}^{\epsilon}}\!\left[\exp\!\left(-\frac{2}{\epsilon}g(\hat{X}^{\epsilon}_{T})+\frac{1}{\epsilon}\int_{t}^{T}|u_{s}(\hat{X}^{\epsilon}_{s})|^{2}\,ds\right)\Big|\,\hat{X}^{\epsilon}_{t} = x\right].
\end{equation}
Then, the dynamic programming is applied to derive the HJB equation. In detail, let 
\begin{equation}
    J(t, \hat{X}, u) = \mathbb{E}_{\mathbb{P}^{\epsilon}, x}\!\left[\exp\!\left(-\frac{2}{\epsilon}g(\hat{X}^{\epsilon}_{T})+\frac{1}{\epsilon}\int_{t}^{T}|u_{s}(\hat{X}^{\epsilon}_{s})|^{2}\,ds\right)\right].
\end{equation}
Since there is an optimal control $u^{*}$ based on (19) in $[t, T]$, we consider 
\begin{equation}
    \tilde{u}_{s} = \left\{\begin{array}{l}
         v_{s}, \quad s\in [t, t+h],  \\
         u^{*}_{s},\quad s\in (t+h, T], 
    \end{array}\right.
\end{equation}
where $v\in \mathcal{U}$. Then, we can derive the so-called Bellman equality. In detail, we have 
\begin{equation}
    \begin{split}
        J(t, \hat{X}, \tilde{u}) &= \mathbb{E}_{\mathbb{P}^{\epsilon}, x}\left[\exp\left(-\frac{2}{\epsilon}g(\hat{X}^{\epsilon}_{T})+ \frac{1}{\epsilon}\int_{t}^{t+h}|v_{s}(\hat{X}^{\epsilon}_{s})|^{2}ds + \frac{1}{\epsilon}\int_{t+h}^{T}|u^{*}_{s}(\hat{X}^{\epsilon}_{s})|^{2}ds\right)\right]\\ 
        & = \mathbb{E}_{\mathbb{P}^{\epsilon}, x}\left[\exp\left(\frac{1}{\epsilon}\int_{t}^{t+h}|v_{s}(\hat{X}^{\epsilon}_{s})|^{2}ds\right)\exp\left(-\frac{2}{\epsilon}g(\hat{X}^{\epsilon}_{T}) + \frac{1}{\epsilon}\int_{t+h}^{T}|u^{*}_{s}(\hat{X}^{\epsilon}_{s})|^{2}ds\right)\right]\\ 
        & = \mathbb{E}_{\mathbb{P}^{\epsilon}, x}\left[\exp\left(\frac{1}{\epsilon}\int_{t}^{t+h}|v_{s}(\hat{X}^{\epsilon}_{s})|^{2}ds\right)\mathbb{E}\left[\exp\left(-\frac{2}{\epsilon}g(\hat{X}^{\epsilon}_{T}) + \frac{1}{\epsilon}\int_{t+h}^{T}|u^{*}_{s}(\hat{X}^{\epsilon}_{s})|^{2}ds\right)\Big|\,\hat{X}^{\epsilon}_{t+h}\right]\right]\\ 
        & = \mathbb{E}_{\mathbb{P}^{\epsilon}, x}\left[\exp\left(\frac{1}{\epsilon}\int_{t}^{t+h}|v_{s}(\hat{X}^{\epsilon}_{s})|^{2}ds\right)V(\hat{X}^{\epsilon}_{t+h}, t+h)\right].
    \end{split}
\end{equation}
Thus, we have 
\begin{equation}
    V(x, t) \leq \mathbb{E}_{\mathbb{P}^{\epsilon}, x}\left[\exp\left(\frac{1}{\epsilon}\int_{t}^{t+h}|v_{s}(\hat{X}^{\epsilon}_{s})|^{2}ds\right)V(\hat{X}^{\epsilon}_{t+h}, t+h)\right].
\end{equation}
Similarly, if we take $v = u^{*}$, we can still get 
\begin{equation}
    V(x, t) = \mathbb{E}_{\mathbb{P}^{\epsilon}, x}\left[\exp\left(\frac{1}{\epsilon}\int_{t}^{t+h}|u^{*}_{s}(\hat{X}^{\epsilon}_{s})|^{2}ds\right)V(\hat{X}^{\epsilon}_{t+h}, t+h)\right].
\end{equation}
Now let $\phi_{s} = \psi_{s}V(\hat{X}^{\epsilon}_{s}, s)$, where $\psi_{s} = \exp\left(\frac{1}{\epsilon}\int_{t}^{s}|v_{r}(\hat{X}^{\epsilon}_{r})|^{2}dr\right)$. It is easy to observe that $\phi_{t} = V(x, t)$ and $\phi_{t} = \mathbb{E}_{\mathbb{P}^{\epsilon}, x}[\phi_{t+h}]$. Using Ito's formula, we can get
\begin{equation}
    d\phi_{s} = V(\hat{X}^{\epsilon}_{s}, s)d\psi_{s} + \psi_{s}dV(\hat{X}^{\epsilon}_{s}, s).
\end{equation}
Moreover, for $\psi_{s}$, we have
\begin{equation}
    d\psi_{s} = \frac{1}{\epsilon}|v_{s}|^{2}\psi_{s}ds,
\end{equation}
and for $V$, based on (24), we have $d\langle\hat{X}^{\epsilon}, \hat{X}^{\epsilon}\rangle_{s} = \epsilon \sigma^{2}(\hat{X}^{\epsilon}_{s}, s)ds$, and thus 
\begin{equation}
\begin{split}
    dV(\hat{X}^{\epsilon}_{s}, s) &= \partial_{s}Vds + \frac{1}{2}\partial_{xx}Vd\langle\hat{X}^{\epsilon}, \hat{X}^{\epsilon}\rangle_{s} + \partial_{x}V d\hat{X}^{\epsilon}_{s}\\ 
    & = \partial_{s}Vds + \frac{\epsilon}{2}\sigma^{2} \partial_{xx}V ds + \partial_{x}V(b - \sigma v)ds + \sqrt{\epsilon}\sigma \partial_{x}VdW_{s}\\ 
    &= \mathcal{A}^{s}ds + \sqrt{\epsilon}\sigma \partial_{x}VdW_{s},
\end{split}
\end{equation}
where $\mathcal{A}^{s} = \partial_{s}V + \frac{\epsilon}{2}\sigma^{2}\partial_{xx}V + \partial_{x}V(b-\sigma v)$.
Combining $(33)$ and (34), we can get
\begin{equation}
    d\phi_{s} = \frac{1}{\epsilon}|v_{s}|^{2}\phi_{s}ds + \psi_{s}\mathcal{A}^{s}ds + \sqrt{\epsilon}\sigma \psi_{s} \partial_{x}VdW_{s},
\end{equation}
which equals to 
\begin{equation}
    \phi_{t+h} - \phi_{t} = \int_{t}^{t+h}\left(\frac{1}{\epsilon}|v_{s}|^{2}\phi_{s} + \psi_{s}\mathcal{A}^{s}\right)ds + \int_{t}^{t+h} \sqrt{\epsilon}\sigma \psi_{s} \partial_{x}VdW_{s}.
\end{equation}
Let $h\to 0$ and take expectation, based on the continuity of the sample path, we can get
\begin{equation}
    \begin{split}
        \lim_{h\to 0^{+}}h^{-1}\mathbb{E}_{\mathbb{P}^{\epsilon}, x}\left[\phi_{t+h}-\phi_{t}\right] &=  \lim_{h\to 0^{+}}h^{-1}\mathbb{E}_{\mathbb{P}^{\epsilon}, x}\left[\int_{t}^{t+h}\left(\frac{1}{\epsilon}|v_{s}|^{2}\phi_{s} + \psi_{s}\mathcal{A}^{s}\right)ds\right]\\ 
        & = \frac{1}{\epsilon}|v_{t}|^{2}\phi_{t} + \psi_{t}\mathcal{A}^{t} \geq 0.
    \end{split}
\end{equation}
Moreover, based on (31), we actually have 
\begin{equation}
    \frac{1}{\epsilon}|u_{t}^{*}|^{2}\phi_{t} + \psi_{t}\mathcal{A}^{t} = \inf_{v\in \mathcal{U}}\;(\frac{1}{\epsilon}|v_{t}|^{2}\phi_{t} + \psi_{t}\mathcal{A}^{t}) = 0,
\end{equation}
which leads to the following HJB equation as 
\begin{equation}
    \inf_{u\in \mathcal{U}}\left(\frac{1}{\epsilon}|u|^{2}V + \partial_{t}V + \frac{\epsilon}{2}\sigma^{2}\partial_{xx}V + (b-\sigma u)\partial_{x}V \right) = 0.
\end{equation}
Thus, we can derive the optimal control is actually
\begin{equation}
    u^{*}(x, t) = \frac{\epsilon\sigma(x,t)\partial_{x}V(x,t)}{2V(x, t)}.
\end{equation}
Then, the original PDE is transformed into 
\begin{equation}
    \left\{
    \begin{array}{l}
         \partial_{t}V(x, t) - \frac{\epsilon(\sigma(x,t)\partial_{x}V(x,t))^{2}}{4V(x,t)} + \frac{\epsilon}{2}\sigma(x,t)^{2}\partial_{xx}V(x,t) + b(x,t)\partial_{x}V(x,t) = 0,  \\
         V(x, T) = \exp(-\frac{2}{\epsilon}g(x)). 
    \end{array}
    \right.
\end{equation}
Next, take the logarithm transform $W^{\epsilon}(x,t) = -\frac{\epsilon}{2}\log V(x,t)$ \cite{hult2024deep}, where 
    \begin{align*}
\partial_t V(t,x) &= -\frac{2}{\varepsilon} V(t,x)\,\partial_t W^\varepsilon(t,x),\\
\partial_x V(t,x) &= -\frac{2}{\varepsilon} V(t,x)\,\partial_x W^\varepsilon(t,x),\\
\partial_{xx} V(t,x) &= -\frac{2}{\varepsilon} V(t,x)\,\partial_{xx} W^\varepsilon(t,x)
  + \frac{4}{\varepsilon^2} V(t,x)\bigl(\partial_x W^\varepsilon(t,x)\bigr)^2,
\end{align*}
the original PDE is further simplified to be 
\begin{equation}
    \left\{
    \begin{array}{l}
         \partial_{t}W^{\epsilon}(x, t) - \frac{(\sigma(x,t)\partial_{x}W^{\epsilon}(x,t))^{2}}{2} + \frac{\epsilon}{2}\sigma^{2}(x,t)\partial_{xx}W^{\epsilon}(x,t) + b(x,t)\partial_{x}W^{\epsilon}(x,t) = 0,  \\
         W^{\epsilon}(x, T) = g(x). 
    \end{array}
    \right.
\end{equation}
Then the optimal control can be represented as 
\begin{equation}
    u^{*}(x,t) = -\sigma(x,t)\partial_{x}W^{\epsilon}(x,t).
\end{equation}
This kind of HJB equation is very similar to the Kolmogorov backward equation. Other ways to derive this equation include the Boué--Dupuis \cite{boue1998variational} formula to derive a variational representation and then solve a min-max problem using dynamic programming. Numerically solving this equation is not easy, especially in high dimensional cases. However, the solution of this PDE can also be represented using the nonlinear Feynman--Kac formula, which involves a set of forward-backward SDEs with $X_{s}, Y_{s}, Z_{s}$ \cite{hult2024deep}. Actually, we can also observe the optimal control is connected with the Doob's $h$ transform \cite{zhang2022koopman} with a different formulation.

When $\epsilon\to 0$, it is clear that the diffusion term disappears. While next we can also show that the probability $\rho$ itself is a solution of such PDEs with different transformations. Let 
\begin{equation}
    \varphi(x, t) = \mathbb{E}_{\mathbb{P}^{\epsilon},x}\left[\exp\left(-\frac{1}{\epsilon}g(X_{T}^{\epsilon})\right)\right],
\end{equation}
then based on the Feynman--Kac formula, we can derive that $\varphi(x,t)$ satisfies the following PDE:
\begin{equation}
\left\{
\begin{array}{l}
    \partial_{t}\varphi(x, t) + b(x,t)\partial_{x}\varphi(x,t) + \frac{\epsilon}{2}\sigma^{2}(x,t)\partial_{xx}\varphi(x,t) = 0,\\
    \varphi(x, T) = \exp(-\frac{1}{\epsilon}g(x)).
\end{array}
\right.
\end{equation}
Using a similar transformation $W^{\epsilon}(x, t) = -\epsilon\log \varphi(x, t)$, we can get that 
\begin{equation}
    \left\{
    \begin{array}{l}
         \partial_{t}W^{\epsilon}(x, t) - \frac{(\sigma(x,t)\partial_{x}W^{\epsilon}(x,t))^{2}}{2} + \frac{\epsilon}{2}\sigma^{2}(x,t)\partial_{xx}W^{\epsilon}(x,t) + b(x,t)\partial_{x}W^{\epsilon}(x,t) = 0,  \\
         W^{\epsilon}(x, T) = g(x). 
    \end{array}
    \right.
\end{equation}
It is obvious that (46) and (42) are exactly the same equation. Therefore, we can now establish the connections between the estimator and second order moment and then establish the log-efficiency of the estimator with optimal control.

As $\epsilon\to 0$, $W^{\epsilon}(x,t)$ in (42) and (46) will both converge to the solution $W(x,t)$ with the following PDE:
\begin{equation}
    \left\{
    \begin{array}{l}
         \partial_{t}W(x, t) - \frac{(\sigma(x,t)\partial_{x}W(x,t))^{2}}{2}  + b(x,t)\partial_{x}W(x,t) = 0,  \\
         W(x, T) = g(x). 
    \end{array}
    \right.
\end{equation}
To this end, when we take the optimal control to construct the proposal distribution $\mathbb{Q}$, we can get 
\begin{equation}
\begin{split}
    &\lim_{\varepsilon \to 0} -\varepsilon 
    \log \mathbb{E}\bigl[e^{-\tfrac{1}{\varepsilon} g(X_{T}^\varepsilon)}\bigr] = W(x, 0) = \gamma_{1},\\ 
    &\lim_{\varepsilon \to 0} -\varepsilon 
    \log \mathbb{E}_{\mathbb{Q}}\bigl[e^{-\tfrac{2}{\varepsilon} g(\tilde{X}_{T}^\varepsilon)}Z^{2}(\tilde{X}^{\epsilon}_{T})\bigr] = 2W(x, 0) = 2\gamma_{1}.
\end{split}
\end{equation}
Therefore, we have 
\begin{equation}
    \limsup_{\epsilon\to 0}\ \epsilon \log R(\epsilon) = 0.
\end{equation}
Thus, the importance sampling estimator is log-efficient. Even though we could not make sure that the prefactors of the large deviation estimator match so that the estimator is error vanishing when $\epsilon \to 0$, this estimator is at least good enough to make accurate predictions when $\epsilon$ is not that small. 

\section{Cross entropy methods for importance sampling}
There are many methods for approximating the importance sampling estimator, i.e, the optimal control $u^{*}$. Firstly, we can directly solve the HJB equation (46) using numerical methods such as DG or finite difference methods. Secondly, we can first write the corresponding PDE as a system of forward backward SDEs and then solve the problem with a least squares Monte Carlo method \cite{hult2024deep}. In this manuscript, we would like to try the cross entropy method, which solves the problem in a variational way.  

In detail, the cross entropy methods seek to approximate the optimal proposal distribution (18) using a candidate distribution by minimizing some distance. In detail, we will typically choose to use the KL divergence to find the distribution numerically, which is equivalent to find the optimal control $u^{*}$. Thus, the problem becomes 
\begin{equation}
    \mathcal{Q}^{u} = \arg\min_{u\in \mathcal{U}} \mathbb{D}_{\text{KL}}[\mathcal{Q}^{*}\| \mathcal{Q}^{u}] = \arg \min_{u\in \mathcal{U}} \int_{\Omega}\log \frac{d\mathcal{Q^{*}}}{d\mathcal{Q}^{u}}\mathcal{Q}^{*}(d\omega).
\end{equation}
We can substitute (18) into the above equation to get
\begin{align}
D_{\mathrm{KL}}[Q^*\|Q^u]
&= \int_{\Omega} \log\frac{dQ^*}{dQ^u}\, Q^*(d\omega) \nonumber\\
&= \int_{\Omega} \Bigg(\log\frac{dQ^*}{dP^\varepsilon}-\log\frac{dQ^u}{dP^\varepsilon}\Bigg)\, Q^*(d\omega) \nonumber\\
&= \mathbb{E}_{Q^*}\!\left[\log L_T^* - \log L_T^u\right] \nonumber\\
&= \mathbb{E}_{P^\varepsilon}\!\left[ L_T^*\big(\log L_T^* - \log L_T^u\big)\right] \nonumber\\
&= \mathbb{E}_{P^\varepsilon}\!\left[ L_T^*\log L_T^*\right]
   - \mathbb{E}_{P^\varepsilon}\!\left[ L_T^*\log L_T^u\right],
\end{align}
where we used the Radon--Nikodym chain rule and the identity
$\mathbb{E}_{Q^*}[\varphi]=\mathbb{E}_{P^\varepsilon}[L_T^*\varphi]$. Moreover, $L_{T}^{u}$ is the density function given by Girsanov's theorem. Therefore, since $\mathbb{E}_{P^\varepsilon}[ L_T^*\log L_T^*]$ does not depend on $u$, we obtain
\begin{align}
Q^u 
\in \arg\min_{u\in \mathcal{U}} D_{\mathrm{KL}}[Q^*\|Q^u]
\quad\Longleftrightarrow\quad
u \in \arg\max_{u\in \mathcal{U}}\ \mathbb{E}_{P^\varepsilon}\!\left[ L_T^*\log L_T^u\right].
\end{align}

To make the optimization problem well-defined and numerically tractable, we restrict to a parametric
family of admissible (Markov) controls
\[
u_t = u_\theta(X_t,t),\qquad \theta\in\Theta\subset\mathbb{R}^m,
\]
and denote by $Q^\theta:=Q^{u_\theta}$ the corresponding path measure.
Consider the baseline diffusion under $P^\varepsilon$ where $\sigma = 1$,
\begin{equation}\label{eq:base-sde}
dX_t = b(X_t,t)\,dt + \sqrt{\varepsilon}\,dW_t,
\end{equation}
and the controlled diffusion under $Q^\theta$,
\begin{equation}\label{eq:ctrl-sde}
dX_t = \big(b(X_t,t)-u_\theta(X_t,t)\big)\,dt + \sqrt{\varepsilon}\,dW_t^\theta.
\end{equation}
The Radon--Nikodym derivative $L_T^\theta=\frac{dQ^\theta}{dP^\varepsilon}$ can be expressed without
explicitly using $dW$ by substituting $dW_t=\big(dX_t-b(X_t,t)\,dt\big)/\sqrt{\varepsilon}$ in the
usual Girsanov formula. This yields the equivalent representation
\begin{equation}\label{eq:girsanov-dx}
\log L_T^\theta
=
-\frac{1}{\varepsilon}\int_0^T u_\theta(X_t,t)^\top dX_t
+\frac{1}{\varepsilon}\int_0^T u_\theta(X_t,t)^\top b(X_t,t)\,dt
-\frac{1}{2\varepsilon}\int_0^T \|u_\theta(X_t,t)\|^2\,dt .
\end{equation}
Since direct sampling under $P^\varepsilon$ is inefficient for rare events, we evaluate the objective
in (52) by importance sampling under a current proposal $Q^{\theta^{(k)}}$.
Using
\[
\mathbb{E}_{P^\varepsilon}[F]
=
\mathbb{E}_{Q^{\theta^{(k)}}}\!\left[
F\,\frac{dP^\varepsilon}{dQ^{\theta^{(k)}}}
\right]
=
\mathbb{E}_{Q^{\theta^{(k)}}}\!\left[
F\,\frac{1}{L_T^{\theta^{(k)}}}
\right],
\]
the maximization in (52) becomes
\begin{equation}\label{eq:CE-objective-Qk}
\theta^{(k+1)}
\in
\arg\max_{\theta\in\Theta}\;
\mathbb{E}_{Q^{\theta^{(k)}}}\!\left[
\frac{L_T^*}{L_T^{\theta^{(k)}}}
\,\log L_T^\theta
\right],
\end{equation}
where $L_T^*$ is given by (18), and $L_T^\theta$ is given by
\eqref{eq:girsanov-dx}.
In practice we approximate \eqref{eq:CE-objective-Qk} with $N$ i.i.d.\ trajectories
$\omega_i=\{X_t^{(i)}\}_{t\in[0,T]}$ sampled from $Q^{\theta^{(k)}}$ and use self-normalized weights
\begin{equation}\label{eq:sn-weights}
\tilde w_i^{(k)}=\frac{w_i^{(k)}}{\sum_{n=1}^N w_n^{(k)}},
\qquad
w_i^{(k)}=\frac{L_T^*(\omega_i)}{L_T^{\theta^{(k)}}(\omega_i)}.
\end{equation}
Then the CE update is the weighted maximum-likelihood problem
\begin{equation}\label{eq:CE-WMLE}
\theta^{(k+1)}
\in
\arg\max_{\theta\in\Theta}
\sum_{i=1}^N \tilde w_i^{(k)}\,\log L_T^\theta(\omega_i).
\end{equation}
Recall from (43)  that the optimal control admits a
feedback form linked to the value function, namely
\begin{equation}\label{eq:u-star-W-link}
u^{*}(x,t) = -\,\sigma(x,t)^{\top}\nabla_x W^\varepsilon(x,t).
\end{equation}
This observation suggests that rather than parameterizing $u$ directly, we may parameterize the value
function and obtain a structured control through differentiation.

More precisely, assume that $W^\varepsilon$ can be approximated by a linear expansion in a chosen
dictionary $\{\varphi_j\}_{j=1}^m$:
\begin{equation}\label{eq:W-approx}
W^\varepsilon(x,t)\approx W_\theta(x,t):=\sum_{j=1}^{m}\theta_j\,\varphi_j(x,t),
\end{equation}
where $\theta\in\mathbb{R}^m$ are unknown coefficients.
Combining \eqref{eq:W-approx} with the relation \eqref{eq:u-star-W-link} yields the induced parametric
control family
\begin{equation}\label{eq:u-from-W}
u_\theta(x,t)
:= -\,\sigma(x,t)^{\top}\nabla_x W_\theta(x,t)
= -\,\sigma(x,t)^{\top}\sum_{j=1}^m \theta_j\,\nabla_x\varphi_j(x,t).
\end{equation}
Therefore, $u_\theta$ is linear in the coefficients $\theta$.
Defining the vector-valued basis functions
\begin{equation}\label{eq:psi-def}
\psi_j(x,t) := -\,\sigma(x,t)^{\top}\nabla_x\varphi_j(x,t)\in\mathbb{R}^d,
\end{equation}
we arrive at the convenient linear parameterization
\begin{equation}\label{eq:linear-control}
u_\theta(x,t)=\sum_{j=1}^m \theta_j\,\psi_j(x,t),
\end{equation}
where $\{\psi_j\}_{j=1}^m$ are prescribed (vector-valued) basis functions determined by the chosen
dictionary $\{\varphi_j\}_{j=1}^m$. Afterwards, substituting \eqref{eq:linear-control} into \eqref{eq:girsanov-dx} shows that $\log L_T^\theta$ is a
concave quadratic function of $\theta$:
\begin{align}
\log L_T^\theta
&=
-\frac{1}{\varepsilon}\sum_{j=1}^m \theta_j
\left(
\int_0^T \psi_j(X_t,t)^\top dX_t
-
\int_0^T \psi_j(X_t,t)^\top b(X_t,t)\,dt
\right)\nonumber\\
&\quad
-\frac{1}{2\varepsilon}\sum_{j,\ell=1}^m \theta_j\theta_\ell
\left(
\int_0^T \psi_j(X_t,t)^\top \psi_\ell(X_t,t)\,dt
\right).
\label{eq:logL-quadratic-dx}
\end{align}
Therefore, the maximizer of \eqref{eq:CE-WMLE} solves the weighted normal equations
\begin{equation}\label{eq:normal-eq-dx}
A^{(k)}\,\theta^{(k+1)} = r^{(k)},
\end{equation}
with
\begin{align}
A^{(k)}_{j\ell}
&=
\sum_{i=1}^N \tilde w_i^{(k)}
\int_0^T \psi_j(X_t^{(i)},t)^\top \psi_\ell(X_t^{(i)},t)\,dt,
\label{eq:A-dx}\\
r^{(k)}_{j}
&=
\sum_{i=1}^N \tilde w_i^{(k)}
\left(
\int_0^T \psi_j(X_t^{(i)},t)^\top b(X_t^{(i)},t)\,dt
-
\int_0^T \psi_j(X_t^{(i)},t)^\top dX_t^{(i)}
\right).
\label{eq:r-dx}
\end{align}
In computations we use a small ridge regularization
$A^{(k)}\leftarrow A^{(k)}+\lambda_{\mathrm{ridge}}I$ to ensure numerical stability.
Let $t_n=n\Delta t$ ($n=0,\dots,M$; $M=T/\Delta t$) and denote $\Delta X_n=X_{n+1}-X_n$.
Simulating \eqref{eq:ctrl-sde} under the current proposal $Q^{\theta^{(k)}}$ yields
\[
X_{n+1}
=
X_n
+
\big(b(X_n,t_n)-u_{\theta^{(k)}}(X_n,t_n)\big)\Delta t
+
\sqrt{\varepsilon}\,\Delta B_n,
\qquad \Delta B_n\sim \mathcal N(0,\Delta t).
\]
The path functional $\log L_T^\theta$ in \eqref{eq:girsanov-dx} is approximated by left-endpoint sums:
\begin{align}
\log L_T^\theta
&\approx
-\frac{1}{\varepsilon}\sum_{n=0}^{M-1} u_\theta(X_n,t_n)^\top \Delta X_n
+\frac{1}{\varepsilon}\sum_{n=0}^{M-1} u_\theta(X_n,t_n)^\top b(X_n,t_n)\,\Delta t
-\frac{1}{2\varepsilon}\sum_{n=0}^{M-1}\|u_\theta(X_n,t_n)\|^2\,\Delta t.
\label{eq:logL-discrete-dx}
\end{align}
Similarly, the integrals in \eqref{eq:A-dx}--\eqref{eq:r-dx} are approximated by
\[
\int_0^T \psi_j^\top \psi_\ell\,dt
\approx
\sum_{n=0}^{M-1}\psi_j(X_n,t_n)^\top \psi_\ell(X_n,t_n)\,\Delta t,
\qquad
\int_0^T \psi_j^\top dX
\approx
\sum_{n=0}^{M-1}\psi_j(X_n,t_n)^\top \Delta X_n,
\]
and the CE iteration is implemented by repeatedly: (i) simulating trajectories under $Q^{\theta^{(k)}}$,
(ii) computing weights via \eqref{eq:sn-weights} using \eqref{eq:logL-discrete-dx},
(iii) assembling $A^{(k)}$ and $r^{(k)}$, and (iv) solving \eqref{eq:normal-eq-dx} for $\theta^{(k+1)}$.

\section{Numerical experiments}

In this section, we would like to consider the following SDE \cite{nusken2021solving}:
\begin{equation}
    dX_{t} = -\nabla V(X_{t})dt + \sqrt{\epsilon}dW_{t},
\end{equation}
where $V(x)$ is the double-well potential defined as 
\begin{equation}
    V(x) = \sum_{i=1}^{d}\kappa_{i}(x_{i}^{2} - 1)^{2}.
\end{equation}
And we would like to approximate the following quantity of interest:
\begin{equation}
    \rho = \mathbb{E}\!\left[\exp\!\left(-\frac{1}{\epsilon}g(X_{T})\right)\right],
\end{equation}
where $g$ is chosen to favor the \emph{right} well (consistent with the discussion below), e.g.
\begin{equation}
    g(x) = \sum_{i=1}^{d}\nu_{i}(x_{i}-1)^{2}.
\end{equation}
In this case, we consider $d = 1$ and $\kappa_{i} = \nu_{i} = 1$. The above SDE is solved by Euler--Maruyama scheme with a step size $\Delta t = 0.001$. We take $X_0=-1$ so that reaching the state near $x=1$ is rare when $\epsilon$ is small.
The control is approximated by 
\begin{equation}
    u(x) = \sum_{m=1}^{J}\theta_{m}\nabla\phi_{m}(x),
\end{equation}
where $\phi_{i}$ is the corresponding RBF kernel defined by 
\begin{equation}
    \phi_{m}(x) = \exp(-\frac{\|x - c_{m}\|^{2}}{2r_{m}^{2}}),
\end{equation}
where $c_{m}, r_{m}$ are the corresponding centers and standard deviations. In practice, we use 17 basis functions to approximate the control, where the centers are specified as $c_{m} = -1.5 + 0.1m, m = 1, \ldots, J$. Moreover, we set $r_{m} = 0.5$ for all basis functions. The number of trajectories to solve the equation is taken to be $N = 30000$. The noise level $\epsilon = 0.05$. In this case, we plot the trajectories without control in Fig.\ref{fig:uncontrolled}. It is clear that without control, it is very rare for the trajectories to reach the state $x = 1$. Thus, the estimated quantity of interest could be highly inaccurate. 

\begin{figure}[htbp]
    \centering
    \includegraphics[width=0.5\linewidth]{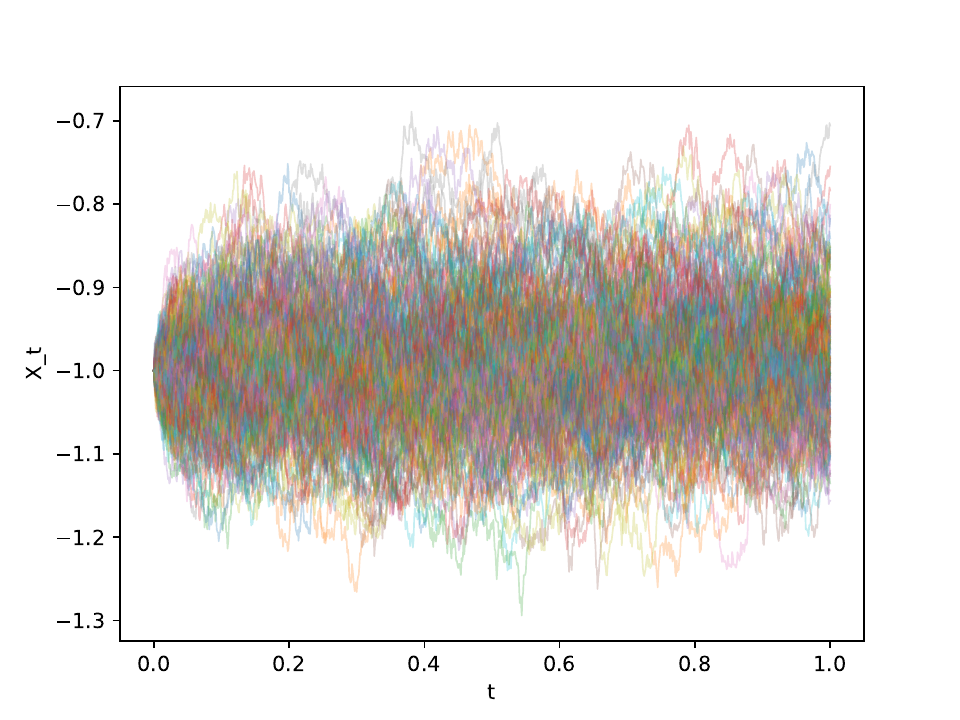}
    \caption{The trajectories with uncontrolled SDE.}
    \label{fig:uncontrolled}
\end{figure}

Next, we also plot the trajectories for the controlled SDEs in Fig.\ref{fig:controlled}. In this case, we observe that most of the trajectories will reach $x = 1$. Thus, we can approximate the quantity with high precision. 
\begin{figure}[t]
    \centering
\includegraphics[width=0.5\linewidth]{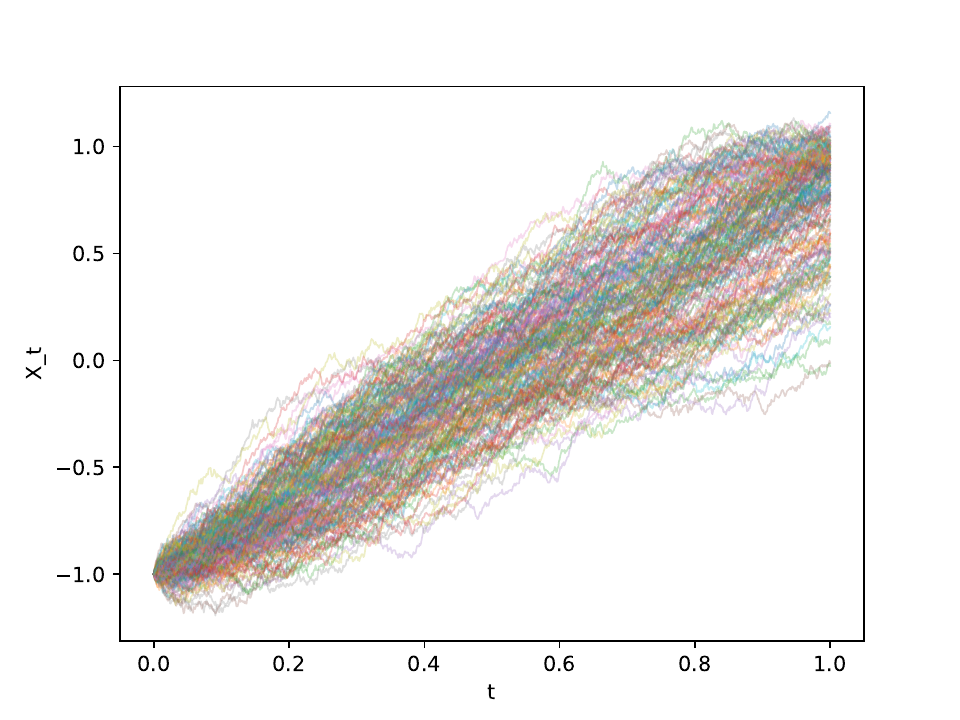}
    \caption{The trajectories for the SDE with learned control.}
    \label{fig:controlled}
\end{figure}
Finally, we plot the modified potential function and the control in Fig.\ref{fig:learned_control} to indicate that it is reasonable why changing the potential could lead to more efficient approximation. This is because by changing the potential function, more trajectories go to the rare areas and the original quantity of interest is no longer a small quantity under the new measure. Finally, the estimated $\hat{\rho}$ is $8.16e-12$, which is much more accurate than the MC estimator. 
\begin{figure}[htbp]
    \centering
    \includegraphics[width=0.5\linewidth]{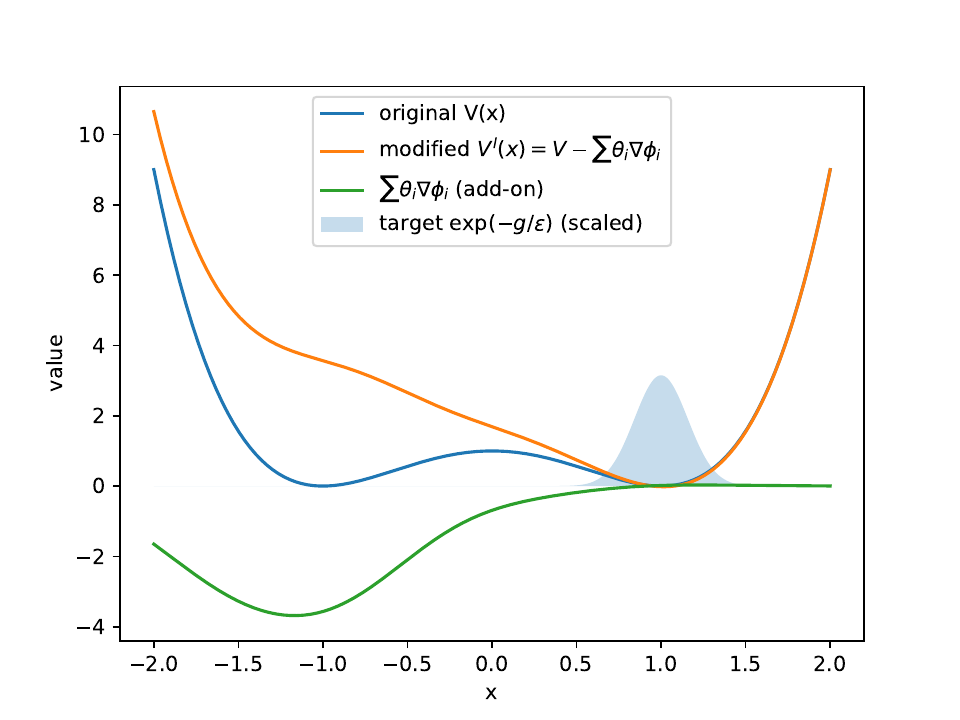}
    \caption{The original potential, modified potential and learned control.}
    \label{fig:learned_control}
\end{figure}

\section{Conclusion}
In this manuscript, we establish the log-efficiency of the importance sampling estimator by leveraging the solution of the associated HJB equation. To approximate the optimal change of measure—equivalently, to approximate the optimal control—we employ the cross-entropy method, which minimizes the Kullback–Leibler divergence by parameterizing the control as a linear combination of prescribed basis functions. We then present a representative numerical example to demonstrate the effectiveness of the proposed approach.

Despite these advantages, several limitations remain. First, in high-dimensional settings the method can suffer from the curse of dimensionality, resulting in prohibitive computational costs. Second, selecting suitable basis functions is challenging in the absence of reliable prior knowledge. Third, the optimal control is generally time-dependent, which complicates the parameterization; moreover, theoretical convergence guarantees for the resulting procedure are not yet established.

\bibliographystyle{plain}
\bibliography{reference}

\end{document}